\documentclass[11pt, a4paper, english]{article}

\usepackage{geometry}

\usepackage{parskip}
\usepackage{amsmath}
\usepackage{amssymb}
\usepackage{amsthm}
\usepackage[utf8]{inputenc}
\usepackage{scrextend}
\usepackage{bbm}
\usepackage{stmaryrd}
\usepackage{lmodern}
\usepackage{physics}
\usepackage{hyperref}
\usepackage{enumerate}
\usepackage{comment}
\usepackage{tikz-cd}
\usepackage{esint} 

\numberwithin{equation}{section}

\newcommand{\R}{\mathbb{R}}
\newcommand{\tM}{\tilde{\mathcal M}}
\newcommand{\tR}{\tilde{\mathcal R}}
\renewcommand{\{}{\left\lbrace}
\renewcommand{\}}{\right\rbrace}
\DeclareMathOperator{\diag}{diag}
\DeclareMathOperator{\id}{id}
\DeclareMathOperator{\Id}{Id}
\DeclareMathOperator{\inj}{inj}
\DeclareMathOperator{\dist}{dist}
\DeclareMathOperator{\sym}{sym}
\DeclareMathOperator{\Vol}{Vol}
\DeclareMathOperator{\Lip}{Lip}

\theoremstyle{plain}
\newtheorem{theorem}{Theorem}[section]
\newtheorem{lemma}[theorem]{Lemma}
\newtheorem{corollary}[theorem]{Corollary}
\newtheorem{proposition}[theorem]{Proposition}

\newtheorem{define}[theorem]{Definition}
\newtheorem{remark}[theorem]{Remark}

\begin{document}
\title{Scaling of the elastic energy of small balls for maps between manifolds with different curvature tensors}
\author{Milan Kr\"omer and Stefan M\"uller}
\maketitle

\tableofcontents

\section{Introduction}
Motivated by experiments and formal asymptotic expansions  in the physics literature \cite{aharoni16},
Maor and Shachar  \cite{maor_shachar19} studied the behaviour of a model elastic  energy of maps between manifolds with
incompatible metrics. For thin objects they analysed  the scaling of the minimal
elastic energy as a function of the thickness. In particular, 
they established the folllowing result.

\begin{theorem}[\cite{maor_shachar19}, Thm 1.1]  \label{th:h4_flat_target}  Let  $(\mathcal M, g)$ be an oriented $n$-dimensional 
Riemannian manifold. Let $p \in M$ and consider a small ball $B_h(p)$ around $p$. For a  map $u$ in the Sobolev space $W^{1,2}(B_h(p); \R^n)$ define the energy
\begin{equation}  \label{eq:energy_h_flat}
E_{B_h(p)}(u) =: \fint_{B_h(p)}  \dist^2(du, SO(g,e)) \, d\Vol_g
\end{equation}
where $SO(g,e)(p')$ denotes the set of orientation preserving isometries from $T_{p'} M$ to $\R^n$ (equipped with  the Euclidean metric $e$ and the standard
orientation) and  where the distance is taken with respect to the Frobenius norm for tensors in $\R^n \otimes T_p^*M$, see   \eqref{eq:dist_in_orthonormal} and \eqref{eq:dist_in_coordinates} below for explicit formulae. For a measure $\nu$ the average with respect to $\nu$ is denoted by $\fint_E f  \, d\nu = (\nu(E))^{-1} \int f \,  d\nu$.

For a tensor $\mathcal A \in T_p M \otimes (T_p^* M)^{\otimes 3}$ define a map $\mathcal B: T_p M \supset B_1(0) \to T_p M \otimes T_p^* M$
by $\mathcal B(X)(Y) = \mathcal A(X,Y,X)$ and an energy
\begin{equation}  \label{eq:energy_limit_A}
\mathcal I_\mathcal A := \min_{ f \in W^{1,2}(B_1(0); T_p M)}   \fint_{B_1(0)}  |\sym df 
-
\frac16 \mathcal B|^2 \, d\Vol_{g(p)}.
\end{equation}
Then 
\begin{equation}  \label{eq:convergence_flat_target}
\lim_{h \to 0} \frac{1}{h^4} \inf E_{B_h(p)} = \mathcal I_{\mathcal R(p)},
\end{equation}
where $\mathcal R(p)$ is the Riemann curvature tensor at $p$. 
\end{theorem}

In   \eqref{eq:energy_limit_A} the norm is the Frobenius norm of tensors in $T_p M \otimes T_p^*M$ and the symmetric 
part of a linear map $L: T_p M  \to T_p M$ is defined by $\sym L = \frac12 (L + L^T)$ where $L^T $ is the adjoint map given 
by $g(p)(L^T X, Y) = g(p)(X, LY)$.

In \cite{maor_shachar19} it is shown that  the quadratic quantity  $\mathcal I_{\mathcal R(p)}$ is actually induced by a scalar product and in particular
$I_{\mathcal R(p)} = 0$ if and only if $\mathcal R(p) =0$.
Recall that by Gauss' theorema egregium, a small ball $B_h(p)$ in $\mathcal M$ can be  mapped  into $\mathbb R^n$ with zero energy $E_{B_h(p)}$  if and only 
if $\mathcal R \equiv  0$ on $B_h(p)$. 

In local coordinates $\mathcal I_{\mathcal A}$ is given as follows. Let  $e_1, \ldots, e_n$   be any
$g(p)$-orthonormal basis of $T_p M$. Then 
\begin{equation}
\mathcal I_\mathcal A = \min_{ \bar f \in W^{1,2}(B_1(0); \R^n)} 
  \fint_{B_1(0)} \sum_{i,k=1}^n  \left(
    \frac12 \left(   \frac{\partial \bar f^i}{\partial x^k} +  
\frac{\partial \bar f^k}{\partial x^i} \right) 
- \frac16  \sum_{j,l=1}^n \mathcal A^i_{jkl} x^j x^l  \right)^2   \,    dx 
\end{equation}
where now $B_1(0)$ is the unit ball in $\R^n$ and 
\begin{equation}
\mathcal A^i_{jkl} = g(p)( e_i, \mathcal A(e_j, e_k, e_l)).
\end{equation}
The functions $f$ and $\bar f$ are related by the identity $\bar f^i(x) = g(p)(e_i, f(  \sum_{j=1}^n x^j e_j))$.

Based on Theorem~\ref{th:h4_flat_target} and heuristic reasoning in the physics literature, Maor and Shachar raise the question whether 
Theorem~\ref{th:h4_flat_target} 
can be generalized to 
non-flat targets with $\mathcal R$ replaced by the difference of the curvature tensors in the target and the domain \cite[Open question 1, p.\ 154]{maor_shachar19}.
Here we show that this is true if the difference of the curvature tensors is properly interpreted.

\begin{theorem}   \label{th:h4_compact_target}  Let $(\mathcal M, g)$ and $(\tM, \tilde g)$ be smooth oriented Riemannian manifolds and suppose that $\tM$
is compact. For $p \in \mathcal M$, $h > 0$ and 
a  map $u$ in the Sobolev space $W^{1,2}(B_h(p); \tM)$ define the energy
\begin{equation} \label{eq:energy_h}
E_{B_h(p)}(u) =: \fint_{B_h(p)}  \dist^2(du, SO(g,\tilde g)) \, d\Vol_g
\end{equation}
where $\dist(du,SO(g,\tilde g))(p')$ denotes the Frobenius distance in $T_{u(p')} \tM \otimes T_{p'}^*M$ of $du(p')$ from
the   set of orientation preserving isometries from $T_{p'} M$ to 
$T_{u(p')} \tM$.
Then 
\begin{equation}  \label{eq:convergence_general_target}
\lim_{h \to 0} \frac{1}{h^4} \inf E_{B_h(p)} = \min_{q \in \tM}  \, \, \min_{Q \in SO(T_p \mathcal M, T_q \tM)} \mathcal I_{\mathcal R(p)- \tilde{\mathcal R}^Q},
\end{equation}
where $\tilde {\mathcal{R}}^Q$ is the pullback of the the Riemann curvature tensor $\tR(q)$ under $Q$, i.e.,
\begin{equation} 
 \tR^Q(X,Y,Z) = Q^{-1} \tR(q)(QX, QY, QZ)
\end{equation}
and where $SO(T_p \mathcal M, T_q \tM)$ denotes the set of orientation preserving isometries from $T_p \mathcal M$ (equipped with the metric $g(p)$) and
$T_q(\tM)$ (equipped with the metric $\tilde g(q)$).
\end{theorem}

The result can be extended to noncompact targets $\tilde{\mathcal{M}}$, if $\tilde{\mathcal{M}}$ satisfies a uniform regularity condition near infinity and
if the minimum over $q$  is replaced by an infimum, see Corollary~\ref{co:noncompact_target} below.
In particular the result holds for the hyperbolic space $\mathbb H_K$ of
constant curvature $K < 0$,  and we recover Theorem~\ref{th:h4_flat_target}  if we take $\tM = \R^n$. 

The heuristic argument for the validity of both theorems  is simple. In normal coordinates  (i.e. those induced by the exponential map) 
in a neighbourhood of $p \in \mathcal M $ and $q = u(p) \in \tilde{\mathcal{M}}$
the metrics behave like $g(v) = \Id + q(v) + \mathcal O(|v|^3)$ and $\tilde g(v) =   \Id + \tilde q(v) +  \mathcal O(|v|^3)$
where $q$ and $\tilde q$ are homogeneous of degree $2$ and determined by the Riemann curvature tensors at $p$ and $q$,
respectively, see \eqref{eq:metric_normal} below.
This suggests to look for approximate minimizers  of the elastic energy of the form
\begin{equation}  u(\exp_p X) =  \exp_q (Q( X + h^3 f(X/h))  \end{equation}
with $Q \in SO(T_p \mathcal M, T_q \tM)$ and $f : T_p M \to T_p M$. 
Then $d (\exp_q^{-1} \circ  u  \circ \exp_p)=  Q (\id + h^2 df)$ and optimization over $f$  and $Q$ should yield the asymptotically optimal
behaviour of the energy. 

Similar to the reasoning in  \cite{maor_shachar19},  the proof of Theorem~\ref{th:h4_compact_target}
 relies on a corresponding $\Gamma$-convergence result where the notion of  convergence
of sequences of maps $u_h : B_h(p) \to  \tilde{\mathcal{M}}$ incorporates a blow-up which reveals the  map $f$. 
One key additional difficulty for non-flat targets is that maps $u_h$ with small energy need not be continuous. Thus $u_h(B_h(p))$ may
not be contained in a single chart of  $  \tilde{\mathcal{M}}$ and we cannot rely on Taylor expansion in exponential coordinates in the target.

To overcome this difficulty,  we define a new notion of convergence of the maps $u_h$ which is based on Lipschitz
approximations and exploits the fact that Sobolev maps agree with Lipschitz maps on a large subset.
The idea to use Lipschitz approximation to treat manifold-valued maps has already been used in 
\cite[pp. 390--391]{kupferman_maor_asaf_reshetnyak}. The use of Lipschitz approximations to define a suitable notion of convergence after blow-up
seems, however, to be new. We believe that this approach might be useful for other problem involving manifold-valued maps, too.

\bigskip

The remainder of this paper is organized as follows. 
In Section~\ref{se:preliminaries} we introduce the relevant notation and definitions, 
in particular the definition of Sobolev maps with values in a Riemannian manifold. 
In Section~\ref{se:convergence} we introduce a new notion of convergence based
on blow-ups of Lipschitz approximations and show that the limit is well-defined, and
in particular does not depend on which Lipschitz approximation is used. 
Based on this convergence notion we establish compactness and $\Gamma$-convergence
results in Section~\ref{se:gamma}. Finally, in Section~\ref{se:energy},   we deduce Theorem~\ref{th:h4_compact_target},
 i.e.  convergence 
of the rescaled energy,   in the usual way from compactness and $\Gamma$-convergence.

\section{Preliminaries}  \label{se:preliminaries}

Here we recall three  facts: the notion of Sobolev spaces of  maps with values in a  Riemannian manifold,
the expression of  $\dist(du, SO(g, \tilde g))$ in local coordinates, and the 
expansion of the metric near the origin in normal  coordinates. 

For the rest of this paper  $(\mathcal M, g)$,  $(\tilde{\mathcal{M}}, \tilde g)$ will always denote   smooth oriented Riemannian $n$-dimensional  manifolds. 
We often drop $g$ or $\tilde g$ in the notation. We denote by $d_g$ the inner metric of $\mathcal M$, i.e. $d_g(p, p')$ is given by the infimum of  the length
of curves connecting $p$ and $p'$.

The Sobolev spaces $W^{1,p}(\mathcal M)$ of functions $v: \mathcal M \to \mathbb R$
are defined by using a partition of unity and local charts. The definition of Sobolev maps with values in $\tilde{\mathcal{M}}$ is more subtle, since
Sobolev maps need not be continuous and hence the image of a small ball in $\mathcal M$ may not be contained in a single chart of  $\tilde{\mathcal{M}}$.
To overcome this difficulty, we use the fact  that $\tilde{\mathcal{M}}$ can be isometrically embedded in some $\mathbb R^s$ if $s$ is chosen sufficiently large. 
We thus may assume that $\tilde{\mathcal{M}} \subset \mathbb R^s$ and  for an open subset $U \subset \mathcal M$ we define
\begin{equation}  \label{eq:define_sobolev}
W^{1,p}(U; \tilde{\mathcal{M}}) = \{ u \in W^{1,p}(\mathcal M; \mathbb R^s) :  \, \,  \text{$u(x) \in \tilde{\mathcal{M}}$ for a.e. $x \in U$}  \}.
\end{equation}
It is easy to check that  for a map $u \in W^{1,p}(U; \tilde{\mathcal{M}})$  the weak differential $du$ (obtained by viewing $u$ 
as a map with values in $\mathbb R^s$) satisfies 
$\mathrm{range} (du(x)) \subset T_{u(x)} \tilde{\mathcal{M}}$ for a.e.\ $x \in U$. 

Equivalently,  one  can define the Sobolev space $W^{1,p}(U; \tilde{\mathcal{M}})$ by  viewing  $\tilde{\mathcal{M}}$
 as a metric space with the inner metric $d_{\tilde g}$ 
and use the theory  of Sobolev spaces with values in a metric space, see, for example, 
 \cite{HKST} and \cite{reshetnyak97}.
 Alternatively, one can use the intrinsic definition Sobolev maps with values in manifolds, introduced
 by Convent and van Schaftingen \cite{CvS16}.

\medskip

We denote by $\R^{n \times n}$ the space of real $n \times n$ matrices and by $O(n) = \{ A \in \R^{n \times n} : A^T A = \Id\}$ and $SO(n) =
\{ A \in O(n) : \det A = 1\}$ the orthogonal and special orthogonal group. On $\R^{n \times n}$ we use the Frobenius norm given by 
\begin{equation}
|A|^2 = \tr A^T A = \sum_{i,j=1}^n A_{ij}^2.
\end{equation}
This norm is invariant under the left and right action of $O(n)$:
\begin{equation} \label{eq:frobenius_invariant}
| R A Q| = |A|  \quad \forall R, Q \in O(n).
\end{equation}
For a (weakly) differentiable map $u$ from an open  subset of $(\mathcal M, g)$ to $(\tM, \tilde g)$ we define
$\dist(du, SO(g, \tilde g))$ as follows. For $p \in \mathcal M$ let $\underline V = (V_1, \ldots V_n)$ be a positively oriented
orthonormal basis of $(T_p \mathcal M, g(p))$, let $\tilde{ \underline V}$ be a positively oriented orthonormal basis
of $T_{u(p)} \tM$ and let $A= (du)_{\underline V, \tilde{\underline V}}$ be the matrix representation of $du(p)$ in these bases, i.e., 
$du(p) V_j =  \sum_{i=1}^n A_{ij} \tilde V_i$. Then
\begin{equation}  \label{eq:dist_in_orthonormal}
\dist(du, SO(g, \tilde g)) := \min_{Q \in SO(n)} | (du)_{\underline V, \tilde{\underline V}} - Q|.
\end{equation}
In view of  \eqref{eq:frobenius_invariant},  the right hand side does not depend on the choice of (positively oriented) orthonormal bases. 
If $\underline X$ and $\tilde{\underline X}$ are general positively oriented bases and if we define matrices
$(g_{\underline X})_{ij} = g(p)(X_i, X_j)$ and $(\tilde g_{  \tilde{\underline X}})_{ij}  = \tilde g(u(p))(\tilde X_i, \tilde X_j)$ then
$V_i =  \sum_{j=1}^n (g_{\underline X})^{-1/2}_{ij} X_j$ and $\tilde V_i = \sum_{j=1}^N (\tilde g_{ \tilde{ \underline X}})^{-1/2}_{ij} X_j$ define orthonormal bases. 
Thus, if $(du)_{  \underline X, \tilde {\underline X}}$ is the matrix representation with respect to $\underline X$ and $\tilde{\underline X}$
we get    
\begin{equation}  \label{eq:dist_in_coordinates}
\dist(du, SO(g, \tilde g)) = \min_{Q \in SO(n)} |   \tilde g_{ \tilde{ \underline X}}^{1/2}  \,  \, (du)_{\underline X, \tilde{\underline X}}  \, \,   g_{\underline X}^{-1/2} \, 
 - Q|.
\end{equation}
In particular we see that  $\dist(du, SO(g, \tilde g))$ behaves natural under pullback. More precisely, if  $\mathcal N$ and $\tilde{\mathcal{N}}$ are oriented $n$-dimensional
manifolds and $\varphi: \mathcal N \to \mathcal M$, $\psi: \tilde{\mathcal{N}} \to \tM$ are smooth orientation-preserving  diffeomorphisms then
\begin{equation}  \label{eq:dist_general_basis}
\dist(du, SO(g, \tilde g)) = \dist (d(\psi^{-1} \circ u \circ \varphi), SO(\varphi^* g, \psi^* \tilde g))
\end{equation}
where $\varphi^*g$ denotes the pullback metric given by $\varphi^*g(x)(X,Y) = g(\varphi(x))(d\varphi X, d\varphi Y)$ and $\psi^*\tilde g$ is given by the analogous expression.

\medskip

Finally we recall the expansion of the metric in local coordinates. Let $p \in \mathcal M$,  let $\underline V = (V_1, \ldots, V_n)$ be an orthonormal basis of 
$(T_p \mathcal M, g(p))$, let $\imath_{\underline V}: \R^n \to T_p M$ be given by $\imath_{\underline V}(x) = \sum_{j=1}^n x^j V_j$,  and let $(e_1, \ldots, e_n)$ denote the
standard basis of $\R^n$.  Then\footnote{Some authors define the Riemann curvature tensor by
$\mathcal R'(W,U,V) = \mathcal R(U,V,W)$ where $\mathcal R(U,V,W)$ is given by  \eqref{eq:riemann_curvature}. 
Then $\mathcal R'(X,Y,X) = \mathcal R(Y,X,X) = - \mathcal R(X,Y,X)$ and thus
$\left((\exp_p \circ \imath_{\underline V})^* g\right)_{ik}(x) = \delta_{ik} - \frac13 \mathcal {R'}^i_{jkl}(p) x^j x^l + \mathcal O(|x|^3)$.}
\begin{equation}  \label{eq:metric_normal}
\left((\exp_p \circ \imath_{\underline V})^* g\right)_{ik}(x) 
:= \left((\exp_p \circ \imath_{\underline V})^* g\right)(x)(e_i, e_k) = \delta_{ik}
 +  \frac13 \mathcal R^i_{jkl}(p) x^j x^l + \mathcal O(|x|^3).
\end{equation}
where $\mathcal R$ is the Riemann curvature tensor, i.e., 
\begin{equation} \label{eq:riemann_curvature}
\mathcal R(U,V,W) = \nabla_U \nabla_V W - \nabla_V \nabla_U W - \nabla_{[U,V]} W
\end{equation}
and 
\begin{equation}  \label{eq:components_riemann}
\mathcal R^i_{jkl}(p) = g(p)(V_i, \mathcal R(p)(V_j, V_k, V_l)).
\end{equation}

\section{A new notion of convergence for  blow-ups}  \label{se:convergence}

In this section we introduce a notion of convergence of blow-ups of a sequence of maps $u_{h_k} : B_{h_k}(p) \to \tM$
which is based on a suitable approximation by Lipschitz maps. We show  in particular
that this notion of convergence does not depend on the precise choice of the approximation.

Let $p \in \mathcal M$. We set $B_h(p) = \{ p' \in \mathcal M : d_g(p, p') < h\}$ where $d_g$ is the inner metric induced by the Riemannian metric  $g$ on $\mathcal M$. 
In  $T_p \mathcal M$ we consider the balls $B_r(0) = \{ X \in T_p M : g(p)(X,X) < r^2\}$. Let $\inj(p)$ denote the injectivity radius, i.e., the supremum of
all $r > 0$ such that the exponential map $\exp_p$ is injective on $B_r(0)$. Then for $h < \inj(p)$ the exponential map is a smooth diffeomorphism
from $B_h(0) \subset T_p \mathcal M$ to $B_h(p) \subset \mathcal M$. 

Using a positively oriented  orthonormal frame $\underline V = (V_1, \ldots V_n)$ of $T_p M$  and the corresponding map $\imath_{\underline V}: \R^n \to T_p M$
given by $\imath_{\underline V}(x) = \sum_{j=1}^n x^j V_j$ we can identify maps $f : B_1(0) \subset T_p \mathcal M \to T_p \mathcal M$ with maps $\bar f : B_1(0) \subset \R^n \to \R^n$
by setting $\bar f = \imath_{\underline{V}}^{-1} \circ f \circ \imath_{\underline V}$. In this way we can define the Sobolev space 
$W^{1,2}(B_1(0),\mathbb{R}^n)$ with $B_1(0) \subset T_p \mathcal M$ and we introduce the following equivalence relation on that space
\begin{equation}   \label{eq:equivalence_relation} 
 f\sim g \qquad \text{if $f-g$ is affine and $D(f-g)$ is
skew-symmetric.}
\end{equation}
Here symmetry of $Df$ is defined using the scalar product $g(p)$. Equivalently, $Df$ is symmetric if and only if $D\bar f$ is symmetric as a map from $\R^n$ to  $\R^n$ with
respect to the standard Euclidean metric. 

For $p \in \mathcal M$ and $q' \in \tM$ we denote by $SO(T_p \mathcal M, T_{q'} \tM)$ the set of orientation preserving linear isometries from $T_p \mathcal M$ to $T_{q'} \tM$ (equipped with the metrics
$g(p)$ and $\tilde g(q')$, respectively). By $\mu$ we denote the standard measure on $\mathcal M$: $\mu(E) = \int_E d\Vol_g$. Recall that for a measure $\nu$ we denote the 
average with respect to $\nu$ by $\fint_E f \, d\nu = (\nu(E))^{-1} \int_E f \, d\nu$.

\begin{define}\label{def:define_convergence_notion}
Let $h_k>0$  with
$\lim_{k\rightarrow\infty}h_k=0$,    let $p \in \mathcal M$, and let $u_k$ be a sequence of   maps
in $W^{1,2}(B_{h_k}(p); \tilde{\mathcal{M}})$.
Let  $q \in \tM$, $Q\in SO(T_p \mathcal M, T_{q} \tM)$, and 
$f\in W^{1,2}(B_1(0),T_p \mathcal M)/ \! \!  \! \sim$ where $B_1(0) \subset T_p \mathcal M$. 

We say that $u_k$ converges to the triple $(q,Q,f)$, if  the following three conditions hold.
  \begin{enumerate}[(i)]
  \item $u_k$ converges to the constant map $q$ in measure, i.e.,
  \begin{equation} \label{eq:convergence_condition_1}
  \lim_{k \to \infty}   \frac{1}{ \mu(B_{h_k}(p))}
      \mu\left( \{
        x\in B_{h_k}(p):  d_{\tilde g}(u_k(x), q) \ge \delta
        \}  \right)=0
  \end{equation}
  for every $\delta > 0$;
  
  \item there exist Lipschitz maps
    $w_k:B_{h_k}(p)\rightarrow\tilde{\mathcal{M}}$ such that
    \begin{equation}\label{eq:uniform_lipschitz_condition}
      \sup_k\Lip w_k<\infty,
    \end{equation}
    \begin{equation}\label{eq:convergence_condition_2}
      \sup_k\frac{1}{h_k^4}\frac{1}{|B_{h_k}(p)|}
      \mu\left( \{
        x\in B_{h_k}(p):w_k(x)\neq u_k(x)
        \}
      \right)<\infty;
    \end{equation}
  \item Set
    \begin{equation}  \label{eq:define_qk}
      q_k:=\exp_q\left( \fint_{B_1(0)} (\exp_q^{-1}\circ w_k \circ \exp_p)(h_k X) \, d\Vol_{g(p)}(X) \right).
    \end{equation}

Then 
   there exist $Q_k\in SO(T_p \mathcal M, T_{q_k} \tM)$, 
   $c_k\in\mathbb{R}^n$,  and an element $\check{f}$ of the equivalence class $f$ such that
    $Q_k\rightarrow Q$ and the maps
    $f_k:B_1(0) \subset T_p M \to  T_p \mathcal M$ defined by
    \begin{equation} f_k(X):=\frac{1}{h_k^2}\{ Q_k^{-1}   \frac{1}{h_k}(\exp_{q_k}^{-1}\circ w_k \circ \exp_p)(h_k X) -  X-c_k \}  \label{eq:define_fk}
    \end{equation}
    satisfy   
    \begin{equation}  \label{eq:converge_fk_def} f_k\rightharpoonup \check{f}\quad\text{ in }
      W^{1,2}(B_1(0),T_p \mathcal M)
    \end{equation}
  \end{enumerate}
   We denote this convergence  by $u_k\rightarrow(q,Q,f)$.
  \end{define}

  \begin{remark}  \begin{enumerate}
  \item To see that the right hand sides of   
 \eqref{eq:define_qk} and \eqref{eq:define_fk} are  well defined for sufficiently large $k$  note 
  that it follows from \eqref{eq:convergence_condition_1}, \eqref{eq:uniform_lipschitz_condition}, and
   \eqref{eq:convergence_condition_2} that 
   \begin{equation}  \label{eq:uniform_convergence_wk}
   \lim_{k \to \infty} \sup_{p' \in B_{h_k}(p)} d_g(w_k(p'), q) = 0.
   \end{equation}  
   
   Hence, for large enough $k$, the set  $w_k(B_{h_k}(x))$ is contained in a ball around $q$
   on which $\exp_q^{-1}$ is defined and  a diffeomorphism.
      Moreover
    \eqref{eq:uniform_convergence_wk}  implies that 
    \begin{equation} \label{eq:convergence_qk}  \lim_{k \to \infty}  d_{\tilde g}(q_k, q) = 0
    \end{equation}  and thus $\exp_{q_k}^{-1} \circ w_k$ is 
   also well-defined for $k$ large enough. 
   
   \item The linear  maps $Q_k$ have different target spaces. To define the convergence $Q_k \to Q$  one uses a local
   trivialization of the tangent bundle $T \tM$. More explicitly,  one can check convergence by expressing $Q_k$ in a smooth local frame, see the proof
   of Lemma~\ref{le:convergence_welldefined}  below.
   
   \item The reader might wonder why we introduce the points  $q_k$ rather than defining $f_k$  simply by using $\exp_q^{-1}$.  The point is
   that the Lipschitz estimate on $w_k$ ensures that the image $w_k(B_{h_k}(p))$ is contained in a ball of radius $C h_k$ around $q_k$. Thus
   in normal coordinates around $q_k$ one can  obtain  estimates
  like  \eqref{eq:metric_normal}
 with error terms of order $\mathcal O(h_k^2)$.
 Normal coordinates around $q$ give only weaker estimates since we know $d_{\tilde g}(q_k, q) \to 0$, but in general there is 
   no rate of convergence in terms of $h_k$. 
   
   \item Instead of the points $q_k$ one can use in  \eqref{eq:define_fk} a more intrinsically defined Riemannian centre of mass
   which depends only the maps $w_k$ and not on $q$. Indeed, the Lipschitz condition on $w_k$ and the fact that the images of the maps  $w_k$ stays
    in a bounded set  of $\tM$ imply that, for sufficiently large $k$, there exists a unique point $\check{q}_k$ which minimizes the quantity
    $D(q') = \int_{B_{h_k}(p)} d^2_{\tilde g}(w_k,q') \, d\Vol_g$, see \cite[Def. 1.3]{karcher77}. We have opted for the more pedestrian definition 
     \eqref{eq:define_qk} because it is simpler and is sufficient for our purposes.
   \end{enumerate}
   \end{remark}

We 
show next  that  if    $u_k\rightarrow(q,Q,f)$, then   $Q$ and   $f$  are uniquely determined by the sequence $u_k$.  In particular, 
 they do not depend on the choices of $w_k$, $Q_k$, and $c_k$.  
 Note that $q$ is determined by $u_k$ in view of 
 \eqref{eq:convergence_condition_1}.
 We also show that $c_k$ is of order $h_k$.

 \begin{lemma}  \label{le:convergence_welldefined} 
 Suppose that $u_k$, $w_k$,  $Q_k$, $c_k$,   $q$, $Q$, $f$,   and $\check f$  are as in 
Definition~\ref{def:define_convergence_notion} and in particular conditions
\eqref{eq:convergence_condition_1}--\eqref{eq:converge_fk_def} 
hold.
Suppose that there exist $w'_k$,  $Q'_k$, $c'_k$, $f'_k$,  $Q'$, $f'$ and $\check{f}'$ such that 
conditions \eqref{eq:uniform_lipschitz_condition}-- \eqref{eq:converge_fk_def}  hold for the primed quantities.
Then $Q' = Q$ and $f' = f$ (as equivalence classes).

Moreover, if conditions (i)--(iii) in Definition~\ref{def:define_convergence_notion} are satisfied, then
\begin{equation} \label{eq:estimate_ck}
\sup_k \frac{|c_k|}{h_k} < \infty.
\end{equation}
\end{lemma}

\begin{proof}
\begin{enumerate}[Step 1:]
\item Estimate for $d_{\tilde g}(q_k,q_k')$.

  Let 
  $\tilde w_k(X)=w_k(\exp_p h_k X),\,\tilde w_k'(X)=w_k'(\exp_p h_k X)$.
  Then, by \eqref{eq:uniform_lipschitz_condition}, 
  \[\Lip\tilde w_k+\Lip\tilde w_k'\leq Ch_k,\]
  and, by \eqref{eq:convergence_condition_2}, 
  \[\mu(\{X\in B_1(0):\tilde w_k(X)  \neq \tilde w_k'(X)\})\leq Ch_k^4.\]
  Thus for each $X\in B_1(0)$ there exists $Y\in B_1(0)$ such that
  $|Y-X|\leq C 
   h_k^{4/n}$ and $\tilde w_k(Y)=\tilde w_k'(Y)$.
  It follows that
  \[     \sup_{x      } |\tilde w_k(X)-\tilde w_k'(X)|\leq Ch_k^{1+4/n},\]
  and
  \[\sup_{ x   \in B_{h_k}(p)}  |\exp_q^{-1}w_k(x)-\exp_q^{-1}w_k'(x)|
    \leq Ch_k^{1+4/n}.
  \]
  Since
  \[\frac{1}{\mu(B_h(p))}\mu\left(
      \{x:\exp_q^{-1}w_k(x)\neq\exp_q^{-1}w_k'(x)\}
    \right)
    \leq Ch_k^4,
  \]
  we get
  \[d_{\tilde g}(q_k,q_k')
    \leq\frac{C}{h_k^n}(Ch_k^{4+n}h_k^{1+4/n})
    \leq Ch_k^{5+4/n}.
  \]
  
  \item Comparison of $\exp_{q'_k}^{-1}$ and $\exp_{q_k}^{-1}$.\\
  Here and in the rest of the argument  it is convenient to work in local coordinates.
   Thus let $\tilde{\underline V} =: (\tilde{\underline V}_1, \ldots, \tilde{\underline V}_n)$ be a smooth, positively oriented,
 orthonormal frame defined in an open neigbhourhood of $q$. For $q'$ in that neighbourhood  consider the isometries 
 $\imath_{\tilde{\underline V}(q')}: \R^n \to  T\tM_{q'}$ given by  $\imath_{\tilde{\underline V}(q')}: = 
 \sum_{j=1}^n y^j \tilde V_j(q')$. 
 Similarly, fix a positively oriented orthonormal basis $\underline V$ of $T_p \mathcal M$ and define $\imath_V$ in the same way.
 
 Recall that $\inj(q)$ denotes the injectivity radius of $\exp_q$. Thus there exists a $\rho > 0$ such that for $\tilde q, \tilde{\tilde{q}} \in B_\rho(q)$
 and $x \in B_{\inj(q)/2}(0) \subset \R^n$  the expression 
 $$
 v_{\tilde q, \tilde{\tilde{q}}}(x) = \left(\imath_{  \tilde{\underline V}(\tilde{\tilde{q}})}^{-1}
  \circ \exp_{ \tilde{\tilde{q}}}^{-1} \circ \exp_{\tilde q} \circ  \imath_{ \tilde{\underline V}(\tilde{q})}\right)(x)
  $$ 
  is well defined and smooth as a map from $ B_\rho(q) \times  B_\rho(q) \times  B_{\inj(q)/2}(0)$ to $\R^n$. Moreover $v_{\tilde q, \tilde q} = \id$. 
  Thus
  \begin{equation} \label{eq:error_exp_qk}
  \|  dv_{\tilde q, \tilde{\tilde{q}}}(x) - \Id \| \le C d_{\tilde g}(\tilde q, \tilde{\tilde{q}})  \qquad \forall \tilde q, \tilde{\tilde{q}} \in B_{\rho/2}(q), \quad  \forall x \in B_{\inj(q)/4}(0).
  \end{equation}
  It follows from  \eqref{eq:convergence_qk}   and Step 1 that the maps  $\bar v_k$ given by 
  \begin{equation}  \label{eq:define_vk}
  \bar v_k(x) =  \frac{1}{h_k}\left( \imath_{  \tilde{\underline V}(q'_k)}^{-1} \circ \exp_{q'_k}^{-1} \circ \exp_{q_k} \circ  \imath_{  \tilde{\underline V}(q_k)}\right)(h_k x)
  \end{equation}
  are well-defined  for sufficiently large $k$ and $x \in B_{\inj(q)/ 2 h_k}$ and satisfy
  \begin{equation}  \label{eq:estimate_vk}
   | d\bar v_k(x) - \Id|  \le C h_k^{5 + 4/n}  \qquad \forall x \in B_{\inj(q)/4 h_k}.
  \end{equation}

 \item Uniqueness of $Q$ and $f$. \\
 Using the frames introduced in Step 2,  we define maps $\bar f_k: B_1(0) \subset \R^n \to \R^n$ and linear maps
 $\bar Q_k: \R^n \to \R^n$ by
 \begin{eqnarray}
 \bar Q_k &=& \imath_{  {\tilde{\underline{V}}(q_k)}}^{-1} \circ Q_k \circ \imath_{\underline V},  \\
 \bar f_k &=&  \imath_{\underline V}^{-1} \circ f_k \circ \imath_{\underline V}, 
 \end{eqnarray}
  and similarly for the primed quantities. We use the analogous definitions for  the limits $Q$ and $\check{f}$  (with $q_k$ replaced by $q$).
 Then $\bar Q_k, \bar{Q'_k} \in SO(n)$ and $Q_k \to Q$ if and only if $\bar Q_k \to \bar Q$. Similarly $f_k \rightharpoonup \check{f}$ in $W^{1,2}$ if and only
 if $\bar f_k \rightharpoonup \bar{\check{f}}$ in $W^{1,2}$.
 
 We also define the following maps from $B_1(0) \subset \R^n$ to $\R^n$:
 \begin{eqnarray}
 \bar w_k(x) &=& \frac1{h_k}  ( \imath_{  {\tilde{\underline{V}}(q_k)}}^{-1}  \circ  \exp_{q_k}^{-1}  \circ w_k \circ \exp_p \circ  \imath_{\underline V})(h_kx),\\
  \bar w'_k(x) &=&  \frac1{h_k} ( \imath_{  {\tilde{\underline{V}}(q'_k)}}^{-1}  \circ  \exp_{q'_k}^{-1}  \circ w'_k \circ \exp_p \circ  \imath_{\underline V})(h_kx), \\
  \tilde w'_k(x) &=&  \frac1{h_k} ( \imath_{  {\tilde{\underline{V}}(q_k)}}^{-1}  \circ  \exp_{q_k}^{-1}  \circ w'_k \circ \exp_p \circ  \imath_{\underline V})(h_kx).
\end{eqnarray}

 Then 
 $$ \bar w'_k = \bar v_k \circ \tilde w'_k,$$
 where $\bar v_k$ is given by  \eqref{eq:define_vk},  and
 \begin{equation} \label{eq:properties_tildew}
 \Lip \bar w_k + \Lip \tilde  w'_k \le  C, \qquad  \mathcal L^n(\{\bar  w_k \ne \tilde w'_k\}) \le C h_k^4. 
 \end{equation}

 It follows from the definitions of $f_k$ and $f'_k$, as well as the definition
 of $\bar v_k$ in \eqref{eq:define_vk} that
  \begin{eqnarray}
  d \bar f_k &=& \frac{1}{h_k^2}  \left( (\bar {Q}_k)^{-1}     d\bar w_k -  \Id  \right),   \label{eq:barfk} \\
  d \bar f'_k &=& \frac{1}{h_k^2}   \left( ({\bar Q}'_k)^{-1}  d ( \bar v_k \circ\tilde  w'_k) -  \Id  \right) \label{eq:barfk_prime}
 \end{eqnarray}
 
 Now we first exploit the second estimate in \eqref{eq:properties_tildew}
 and the estimate  \eqref{eq:estimate_vk} for $d \bar v_k - \Id$ to show that $\bar Q_k$ and $\bar Q'_k$ have the same limit.
 Let $E_k = \{ \tilde w_k \ne \tilde w'_k\}$. Then $d\bar w_k = d \tilde w'_k$ a.e. in $B_1(0)\setminus E_k$.
Thus, by   \eqref{eq:estimate_vk} and the estimates of the Lipschitz constants in  \eqref{eq:properties_tildew}, we get
\begin{equation}
| d (  \bar v_k \circ\tilde  w'_k) - d \bar w_k)   | \le C h_k^{5 + 4/n} \quad \text{a.e.\ in $B_1(0) \setminus E_k$.}
\end{equation}
Let $\bar R_k := \bar Q_k^{-1} \bar Q'_k$, multiply \eqref{eq:barfk_prime} by $-\bar R_k$, add \eqref{eq:barfk},  and multiply the 
resulting equation by $ h_k^2(1 - 1_{E_k})$. This yields
\begin{equation}  \label{eq:difference_Rk_Id}
h_k^2 ( d\bar f_k - \bar R_k  d \bar f'_k)   (1- 1_{E_k}) = \mathcal O(h_k^{5 + 4/n}) + (\bar R_k- \Id) (1- 1_{E_k}).
\end{equation}
Since $\bar f_k$ and $\bar f'_k$ converge weakly in $L^2$, $\bar R_k \in SO(n)$,   and $\mathcal L^n(E_k) \to 0$, it follows that 
$|\bar R_k - \Id| \le C h_k^2$. In particular,  $\bar R_k \to \Id$ as $h_k \to 0$ and hence $\bar Q = \bar Q'$. 

To show that   $\check{f} \sim \check{f'}$, we note that   there exists a subsequence $k_j \to \infty$ such that the limit
$$ A := \lim_{j \to \infty} \frac{\bar R_{k_j} - \Id}{h_{k_j}^2}$$
exists.
Since $\bar R_k \in SO(n)$,  it follows that $A$ is skewsymmetric.
Dividing    \eqref{eq:difference_Rk_Id}by $h_k^2$ and passing to the limit along the subsequence $k_j$,  we get
$d\bar{\check{f}} - d\bar{\check{f'}}= A$. Thus $\bar{\check{f}} \sim \bar{\check{f'}}$. This is equivalent to $\check{f} \sim \check{f'}$ or
$f = f'$ (as equivalence classes).

\item Proof of \eqref{eq:estimate_ck}.\\
It follows from the definition of $q_k$ and the Lipschitz bound on $w_k$ that $w_k(B_{h_k}(p))$ is contained in a ball $B_{C h_k}(q_k)$.
Thus Taylor expansion of $\hat v_k = \exp_{q_k}^{-1} \circ \exp_q$ around $Z_k = \exp_q^{-1}(q_k)$ yields
$$ \exp_{q_k}^{-1} \circ w_k = \hat v_k \circ \exp_{q}^{-1} \circ w_k = 0 + d\hat v_k(Z_k)[\exp_q^{-1} \circ w_k - Z_k] + \mathcal O(h_k^2).$$
Hence
\begin{eqnarray}
& & \fint_{B_1(0)}  (\exp_{q_k}^{-1} \circ w_k  \circ \exp_p)(h_k X)   \, d\Vol_{g(p)}(X)   \nonumber \\
&=& d\hat v_k(Z_k)\left[  \fint_{B_1(0)}  \left( (\exp_{q}^{-1} \circ w_k  \circ \exp_p)(h_k X)  - Z_k  \right) \, \, \,  d\Vol_{g(p)}(X)   \right] + \mathcal{O}(h_k^2)   \nonumber  \\
 &=&  \mathcal{O}(h_k^2)    \label{eq:avg_qk}
\end{eqnarray}
where we used the definition  \eqref{eq:define_qk} of $q_k$ for the last identity.  
 Since $f_k$ is bounded in $L^2$, equation  \eqref{eq:estimate_ck} now follows by integrating 
 \eqref{eq:define_fk} over $X \in B_1(0)$ and using  \eqref{eq:avg_qk}.
   \end{enumerate}
\end{proof}

\section{Compactness and $\Gamma$-convergence}  \label{se:gamma}

For $u_h\in W^{1,2}(B_h(p),\tilde{\mathcal{M}})$ define the
\emph{energy of $u_h$} by
\[E_{B_h(p)}(u_h)
  :=\fint_{B_h(p)}\dist^2(du_h,SO(g,\tilde g))d\Vol_g.
\]

For points $p \in \mathcal M$ and $q \in \tM$, an orientation preserving isometry $Q \in SO(T_p \mathcal M, T_q \tM)$,  and the unit ball $B_1(0)$ in $T_p \mathcal M$ 
we define a functional $\mathcal I^{q,Q}  : W^{1,2}(B_1(0); T_p \mathcal M) \to \R$ by
\begin{equation} \label{eq:limit_functional}
\mathcal I^{q,Q}(f) = \fint_{B_1(0)} 
 \left| \sym df(X) - \mathcal B(X)\right|^2 \, d\Vol_{g(p)}(X),
\end{equation}
where $|\cdot|$ denotes the Frobenius norm on $T_p \mathcal M  \otimes T_p^*\mathcal M$  and  $\mathcal B(X)$ is the element of $T_p \mathcal M  \otimes T_p^*\mathcal M$
given by 
\begin{equation}  \label{eq:define_B_limit}
\mathcal B(X)(Y) = \frac16 \left(  \mathcal R(p)(X,Y,X) - \tR^Q(X,Y,X) \right)
\end{equation}
with
 \begin{equation} \label{eq:define_RQ_limit}
 \tR^Q(X,Y,X) :=  Q^{-1} \tR(q)(QX, QY, QX).
\end{equation}

It follows directly from the definition that  $\mathcal I^{q,Q}$ depends only on the equivalence
class of $f$ (where the equivalence relation is given by   \eqref{eq:equivalence_relation}). We will thus 
view $\mathcal I^{q,Q}$ also as a functional on the space $ W^{1,2}(B_1(0); T_p \mathcal M)/ \sim$ without change of 
notation.

Our main result is the following compactness and $\Gamma$-convergence result.

\begin{theorem}   \label{th:main} Let $(\mathcal  M, g)$ and $(\tM, \tilde g)$
  be smooth, oriented, $n$-dimensional Riemannian manifolds. Then 
    the following assertions hold:
\begin{enumerate}[(i)]
\item   \label{it:main1} Compactness: 
Assume in addition that   $\tM$ is compact. Let $h_k \to 0$ and assume that  there exists a constant $C > 0$ such the maps
$u_k: B_{h_k}(p)\to \tM$ satisfy  
$E_{h_k}(u_k)\leq Ch_k^4$. Then there exists a
  subsequence $h_{k_j} \to 0$  such that
  \[u_{k_j}\longrightarrow(q,Q,f)\]
  in the sense of  Definition~\ref{def:define_convergence_notion};
\item    \label{it:main2}  $\Gamma-\liminf$ inequality: if $h_k \to 0$ and $u_k \to (q,Q,f)$, then
  \[\liminf_{k\rightarrow\infty}\frac{1}{h_k^4}
    E_{h_k}(u_k)\geq  \mathcal I^{q,Q}(f).
    \]
\item     \label{it:main3} Recovery sequence: Given a triple $(q,Q,f)$ and
  $h_k\rightarrow 0$, there exists $u_k$ such that
  $u_k\rightarrow(q,Q,f)$ and 
  $$ \lim_{k\rightarrow\infty}\frac{1}{h_k^4}
    E_{h_k}(u_k) =    \mathcal I^{q,Q}(f).
  $$
\end{enumerate}
\end{theorem}

The combination of properties (ii) and (iii) can be stated concisely as the fact that  $\frac{1}{h^4} E_h$ $\Gamma$-converges  (with respect to the convergence
in Definition~\ref{def:define_convergence_notion})  to $\mathcal I$ with $\mathcal I(q,Q,f) = \mathcal I^{q,Q}(f)$.

\medskip

To prove compactness, we use the following result on Lipschitz approximation of $\R^s$-valued Sobolev maps. This is a 
minor variation of the classical result by Liu
\cite[Thm. 1]{liu77},  see also \cite[Section 6.6.3, Thm. 3]{evans_gariepy92}.

\begin{lemma}[\cite{fjm02}, Prop. A.1]\label{lem:approx_lipschitz}
  Let $s,n\geq 1$ and $1\leq p<\infty$ and suppose
  $U\subset\mathbb{R}^n$ is a bounded Lipschitz domain. 
  Then there exists a constant $ C = C(U,n,s,p)$ with the following property:

  For each $u\in W^{1,p}(U,\mathbb{R}^s)$ and each $\lambda>0$ there
  exists $v:U\rightarrow\mathbb{R}^s$ such that
  \begin{enumerate}[(i)]
  \item  $\Lip v \le C \lambda$,
  \item 
 $ \mathcal L^n\left( \{x\in U:u(x)\neq v(x)\} \right)
    \leq\displaystyle\frac{C}{\lambda^p}
    \displaystyle\int_{\{x\in U:|du|_e>\lambda\}}|du|_e^p\,dx$.
  \end{enumerate}
  Here $|\cdot|_e$ denotes the Frobenius norm with the respect to the standard scalar products on $\R^n$ and $\R^s$. 
\end{lemma}

\begin{remark}   \label{re:dilation}It is easy to see that the constant $C(U,n,s,p)$ can be chosen invariant under dilations of $U$, i.e.,
$C(rU,n,s,p) = C(U,n,s,p)$. Indeed, given $u \in W^{1,p}(rU, \R^s)$ apply the lemma to the rescaled function $\tilde u:U \to \R^s$ given by $\tilde u(x) = r^{-1}u(rx)$,
obtain a Lipschitz approximation $\tilde v: U \to \R^s$ and define $v$ by $v(y) = r \bar v(y/r)$. 
\end{remark}

\begin{proof}[Proof of Theorem~\ref{th:main}~\eqref{it:main1} (compactness)]
We proceed in two steps. First we show that 
there exists a good Lipschitz approximation $w_k$  of $u_k$ and then deduce compactness by expressing  $\dist(dw_k, SO(g, \tilde g))$
in terms of normal coordinates in $\mathcal M$ and $\tM$.

\begin{enumerate}[Step 1:]
\item Lipschitz approximation: There exists a constant $C > 0$ and Lipschitz maps $w_k: B_{h_k}(p) \to \tM$ such that, for all sufficiently large $k$,
\begin{eqnarray}
\Lip w_k &\le& C,  \label{eq:lipwk}\\
\frac{1}{\mu(B_{h_k}(p))} \mu(\{ u_k \ne w_k \}) & \le & C h_k^4.
\label{eq:uk_not_wk}
\end{eqnarray}

The construction of the maps $w_k$ is very similar to the construction in \cite[pp. 390--391]{kupferman_maor_asaf_reshetnyak}. We include the
details for the convenience of the reader.
To construct $w_k$, we recall that in view of  the Nash imbedding theorem \cite[Theorem 3]{nash56}, we can view $\tilde{\mathcal{M}}$ as
a subset of $\mathbb{R}^s$ for large $s$, with the metric on the tangent space of  $\tM$ induced by the 
Euclidean metric of $\R^s$. Let $\underline V = (V_1,  \ldots, V_n)$ be a positively oriented, orthonormal basis of $T_p \mathcal M$ and define
$\hat u_k: B_{h_k}(0) \subset \R^n \to \tM \subset \R^s$ by
$$ \hat  u_k = u_k \circ \exp_p \circ \imath_{\underline V}$$
were $\imath_{\underline V}(x) = \sum_{j=1}^n x^j V_j$. Let $(\bar g)_{ij} = \big(( \exp_p \circ \imath_{\underline V})^* g\big)(e_i, e_j)$ be the coefficients of the pullback metric in the standard
Euclidean basis. Then by  \eqref{eq:metric_normal}
\begin{equation}  \label{eq:pullback_metric_unrescaled} 
 |\bar g_{ij} - \delta_{ij}| \le C h_k^2  \quad \text{on $B_{h_k}(0)$.}
\end{equation}
Since the Frobenius norm of a map in $SO(n)$ is $\sqrt n$ and since $\tM$ is isometrically imbedded into $\R^s$ it follows from   \eqref{eq:pullback_metric_unrescaled} 
that
\begin{equation}
|d \hat u_k|_e  \le (1 + C h_k^2) \big(\sqrt n + \dist(du_k, SO(g, \tilde g))\big)
\end{equation}
In particular for sufficiently large $k$ we have
\begin{equation}
|d \hat u_k|_e  \ge 4 \sqrt n \quad \Longrightarrow \quad \dist(du_k, SO(g, \tilde g)) \ge \frac12 |d \hat u_k|_e \ge 2 \sqrt n.
\end{equation}
Now apply Lemma~\ref{lem:approx_lipschitz} and Remark~\ref{re:dilation} with $u = u_k$, $U = B_{h_k}(0)$ and $\lambda = 4 \sqrt n$. Denote the corresponding 
Lipschitz approximation by $\hat v_k$ and set $E^2_k = \{ x \in B_{h_k}(0) : \hat v_k \ne \hat u_k\}$. 
Then 
\begin{equation}
\Lip \hat v_k \le C.
\end{equation}
Using that, in addition,  $\det \bar g(x) \ge (1 + Ch_k^2)^{-1}  \ge \frac12$ we get
\begin{eqnarray}
\mathcal L^n(E^2_k) &=&  \frac{C}{\lambda^2} \int_{ \{ x \in B_{h_k}(0) : |d\hat u_k|_e \ge \lambda \}}   |\hat u_k|_e^2 \, dx    \nonumber \\
& \le &  \frac{C}{\lambda^2} \int_{ B_{h_k}(p)}   \dist^2(u_k, SO(g, \tilde g)) \, d\Vol_g  \nonumber   \\
& \le &  C \mu(B_{h_k}(p)) h_k^4.   \label{eq:measure_vk}
\end{eqnarray}

In general,  the map $\hat v_k$ takes values in $\R^s$ rather than in $\tM$. This difficulty can be easily overcome by projecting back to $\tM$. 
Indeed, since $\tM$ is compact, there exists a $\rho > 0$ and a smooth projection $\pi_{\tM}$ from a $\rho$-neighbourhood of $\tM$ in $\R^s$ to $\tM$. 
Now by  \eqref{eq:measure_vk},
 there exists an $x' \in B_{h_k}(0)$ such that $\hat v_k(x')  = \hat u_k(x') \in \tM$. Since the distance function is $1$-Lipschitz we deduce that
 $\dist(\hat v_k(x), \tM) \le C |x-x'| \le C h_k$ for all $x \in B_{h_k}(0)$. Then  $\hat w_k := \pi_{\tM} \circ \hat v_k$ is well-defined for sufficiently large $k$
 and satisfies $\Lip \hat w_k \le C$. Since $\pi|_{\tM} = \id$ we have $\{ \hat w_k \ne \hat u_k \} \subset \{ \hat v_k \ne \hat u_k\}$. 
 Finally,  using that $\exp_p \circ \imath_V$, is Bilipschitz in a neighbourhood of $0$, we see that $w_k := \hat w_k \circ (\exp_p \circ \imath_V)^{-1}$ 
 satisfies  \eqref{eq:lipwk} and \eqref{eq:uk_not_wk}.

\item Compactness\\
The estimate  $\Lip w_k \le C$  implies that the image of $w_k$ is contained in the ball $B(w_k(p), C h_k)$.  Since $\tM$ is compact,
there exists a subsequence $k_j \to \infty$ and $q \in \tM$  such that $w_{k_j}(p) \to q$ as $j \to \infty$.  
Hence $\lim_{j \to \infty} \sup_{B_{h_{k_j}}} d_{\tilde g}(w_{k_j}, q) = 0$ and 
in view of \eqref{eq:uk_not_wk} we get, for all $\delta > 0$, 
$$ \lim_{j \to \infty} \frac{1}{\mu(B_{h_{k_j}}(p))} \mu\left( \{ p' \in B_{h_{k_j}}(p)) : d_{\tilde g}(u_{k_j}(p'), q) \ge \delta\}\right) = 0.$$
Thus condition (i) in Definition~\ref{def:define_convergence_notion} is satisfied for the subsequence $k_j$. 
Condition (ii)  in Definition~\ref{def:define_convergence_notion} is equivalent to  \eqref{eq:lipwk} and \eqref{eq:uk_not_wk}.

To verify condition  (iii) in Definition~\ref{def:define_convergence_notion}, consider the points $q_{k_j}$ defined by  
$$ q_{k_j} := \exp_q \left( \fint_{B_1(0)} (\exp_q^{-1} \circ w_{k_j}  \circ  \exp_p)(h_{k_j} X) \, d\Vol_{g(p)}(X) \right).$$
Since $\exp_q$  and $\exp_p$ are  Bilipschitz with Bilipschitz constant close to one in a small neighbourhood of the origin,  it follows
that  $q_{k_j} \to q$ as $j \to \infty$ and that the image of $w_{k_j}$ is contained in $B_{2 C h_{k_j}}(q_{k_j})$ for $j$ sufficiently large.

Note also that 
the approximation properties \eqref{eq:lipwk} and \eqref{eq:uk_not_wk} in combination with the hypothesis
$E_{B_{h_k}(p)}(u_k) \le C h_k^4$ imply that
\begin{equation}  \label{eq:energy_wk} 
\fint_{B_{h_k}(p)}  \dist^2(dw_k, SO(g, \tilde g)) \, d\Vol_g  \le C h_k^4.
\end{equation}

Now it is convenient to work in local coordinates, as in the proof of Lemma~\ref{le:convergence_welldefined}.
To simplify the notation,  we write  $w_k$ instead of $w_{k_j}$. 
Consider again the maps $\bar w_k: B_1(0) \subset \R^n \to \R^n$ given by
\begin{equation}
 \bar w_k(x) =  \frac1{h_k} ( \imath_{  {\tilde{\underline{V}}(q_k)}}^{-1}  \circ  \exp_{q_k}^{-1}  \circ w_k \circ \exp_p \circ  \imath_{\underline V})(h_k x).\\
\end{equation}
We now apply first the formula  \eqref{eq:dist_general_basis}  for $\dist(du, SO(g, \tilde g))$ with 
$\varphi_k(x) = \exp_p \circ  \imath_{\underline V}(h_k x)$ and $\psi_k(x) = \exp_q \circ  \imath_{\tilde{\underline V}(q_k)}(h_k x)$ and then   \eqref{eq:dist_in_coordinates}.
This yields
\begin{equation}  \label{eq:dwk_local_compactness}
 \dist(dw_k (\varphi(x)), SO(g, \tilde g))  = 
 \dist\left( \big(   \tilde{\bar g}^{(k)} \circ \bar w_k(x)\big)^{1/2} 
  \, d\bar w_k \, 
 \big(\bar g^{(k)}\big)^{-1/2}(x), SO(n)\right),
 \end{equation}
where $\bar g^{(k)}$ is the metric  (expressed in the standard basis of $\R^n$) obtained from the metric  $g$ on $\mathcal M$ by pullback under 
$\varphi_k$ and similarly for ${\tilde{\bar g}}^{(k)}$.

Using the expansion   \eqref{eq:metric_normal} of the metric in normal coordinates and Proposition~\ref{pr:distance_son_perturbed} below we deduce that 
\begin{equation}
\dist(  d\bar w_k, SO(n))(x) \le (1 + C h_k^2) \dist(dw_k, SO(g, \tilde g))\big(\exp_p \circ \imath_{\underline V}(h_k x) \big)+ C h_k^2.
\end{equation}
In view of  \eqref{eq:energy_wk}  this implies that
\begin{equation}
\fint_{B_1(0)} \dist^2\big( d\bar w_k, SO(n)\big)  \, dx \le C h_k^4.
\end{equation}
By the rigidity estimate in  \cite[Thm. 3.1]{fjm02} there exists a constant rotation $\bar Q_k \in SO(n)$ such that
\begin{equation}
\fint_{B(0,1)}    \left| \bar Q_k^{-1}  d\bar w_k - \Id  \right|^2  \, dx \le C h_k^4.
\end{equation}
Thus there exists $\bar c_k \in \R^n$ such that the functions
$$ \bar f_k = \frac{1}{h_k^2} \left(    \bar Q_k^{-1}  \bar w_k - \id -   \bar c_k\right)$$
are bounded in $W^{1,2}(B_1(0); \R^n)$ and hence a subsequence converges weakly in\\
$W^{1,2}(B_1(0); \R^n)$. Unwinding definitions, we see that condition (iii) in 
Definition~\ref{def:define_convergence_notion} is satisfied.
\end{enumerate}
\end{proof}

\begin{proposition} \label{pr:distance_son_perturbed}
Let $A,B, F \in \R^{n \times n}$ and assume that $A$ and $B$ are invertible. 
Then
\begin{align} \label{eq:distance_son_perturbed}
& \, \dist(F, SO(n))   \nonumber \\
   \le  & \,  (1 + |A^{-1} - \Id|) (1 + |B^{-1} - \Id|) \dist(AFB, SO(n))   \nonumber \\
&\,  + |A^{-1} - \Id| + |B^{-1} - \Id| +  |A^{-1} - \Id|  |B^{-1} - \Id|.
\end{align}
For $A = \diag(a^{-1}, 1, \ldots, 1)$, $B = \diag(b^{-1}, 1, \ldots, 1)$, $F = \diag(abc, 1, \ldots, 1)$,
with $a,b,c > 1$ equality holds.
\end{proposition}

\begin{proof} There exist $Q \in SO(n)$ such that $\dist(AFB, SO(n)) = |AFB - Q|$. 
Set $A_Q = Q^{-1} A Q$ and $\bar F_Q = Q^{-1} F$.  Then 
$ \dist(AFB, SO(n))  = |A_Q \bar F_Q B - \Id|$
and 
\begin{align}
|F-Q| = & \, |\bar F_Q  - \Id| \le  |\bar F_Q  - A_Q^{-1}  B^{-1}| + |A_Q^{-1} B^{-1} - \Id|   \nonumber \\
= & \, |A_Q^{-1}( A_Q \bar F_Q B - \Id) B^{-1}| +  |A_Q^{-1}B^{-1} - \Id|    
\end{align}
Now expand $B^{-1}$ and $A_Q^{-1}$ as  $B^{-1} =  \Id + (B^{-1} - \Id) $ and $A_Q^{-1} = Q^{-1} A^{-1} Q =   \Id + Q^{-1} (A^{-1} - \Id) Q$
and use that $|XY| \le |X| \, |Y|$ and $|Q^{-1} (A^{-1} - \Id) Q| = |A^{-1} - \Id|$.
\end{proof}

\begin{proof}[Proof of Theorem~\ref{th:main}~\eqref{it:main2} ($\Gamma-\liminf$ inequality)]
Let $\underline V =(V_1, \ldots, V_n)$ be a positively oriented orthonormal basis of $T_p \mathcal M$ and set 
$\tilde{ \underline V}_k = (Q_k V_1, \ldots, Q_k V_n)$. Then $\tilde{ \underline V}_k$ is a positively oriented orthonormal  basis 
of $T_{q_k} \tM$. Set $\varphi_k(x) = (\exp_p \circ \imath_{\underline V})(h_k x)$ and
 $\psi_k(x)= (\exp_{q_k} \circ \imath_{\tilde{\underline V}_k})(h_k x)$.
 Let $w_k$ be as in Definition~\ref{def:define_convergence_notion} 
and define
\begin{equation}  \label{eq:define_barwk_gammaliminf}
\bar w_k := \psi_k^{-1} \circ w_k \circ \varphi_k,  \qquad \bar E_k := \{ x:  u_k \circ \varphi_k(x)  \ne w_k \circ \varphi_k(x)\}.
\end{equation}
Then $\mathcal L^n( \bar E_k) \le C h_k^4$ and
\begin{equation}  \label{eq:lower_bound_Eh_wk}
 E_{h_k}(u_k) \ge \fint_{B_1(0)} 1_{B_1(0) \setminus \bar E_k}(x) \,  \dist^2(dw_k(\varphi_k(x)), SO(g, \tilde g)) \, d\Vol_{\varphi_k^*g}(x).
\end{equation}

Since the functions $w_k$ satisfy a uniform Lipschitz bound,  we can obtain the lower bound by expressing  $\dist^2(dw_k(\varphi_k(x)), SO(g, \tilde g))$ in normal coordinates at $p$ and $q_k$
and using Taylor expansion on the large set where $dw_k$ is close to $SO(g, \tilde g)$. Specifically,
  using  \eqref{eq:dwk_local_compactness} we get 
\begin{equation}  \label{eq:dwk_local_compactness2}
 \dist(dw_k (\varphi(x)), SO(g, \tilde g))  = \dist\left( \big(\tilde{\bar g}^{(k)}\big)^{1/2}(\bar w_k(x)) \, d\bar w_k \,   \big( \bar g^{(k)}\big)^{-1/2}(x), SO(n) \right),
 \end{equation}
where $\bar g^{(k)}$ is the metric obtained from $g$ by pullback under $\varphi_k$ and similarly for $\tilde{\bar g}$.
The expansion   \eqref{eq:metric_normal} of the metric in normal coordinates yields
\begin{eqnarray}
\bar g^{(k)}_{im}(x) &=& h_k^2 \left( \delta_{im} - \frac13   \sum_{j,l=1}^n g(p)(V_i, \mathcal R(p)(V_j, V_m, V_l) ) \, h_k^2  \, x^j x^l  + \mathcal O(h_k^3 |x|^3)  \right),\\
\tilde{\bar g}_{im}^{(k)}(y) &=&  h_k^2 \left( \delta_{im} - \frac13   \sum_{j,l=1}^n g(q_k)(Q_k V_i, \tR(q_k)(Q_kV_j,  Q_k V_m, Q_k V_l) ) \,  h_k^2 \,  y^j y^l \right.  \nonumber \\  
& &  \qquad  \left. + \,  \mathcal O(h_k^3 |y|^3)  \right).
\end{eqnarray}
Moreover,  it follows from the definition of $\bar w_k$ and $f_k$ that
$$ \bar w_k = \id + h_k^2  \, \imath_{\underline V}^{-1} \circ f_k \circ  \imath_{\underline V} +  \imath_{\underline V}^{-1}  c_k.$$
Now by   \eqref{eq:estimate_ck} we have $c_k \to 0$. Since $f_k$ is bounded in $L^2$ it follows that $\bar w_k \to \id$ in $L^2$. In view of the uniform Lipschitz bound on
$\bar w_k$ we see that $\bar w_k \to \id$ uniformly.
Thus 
\begin{equation}  \label{eq:define_Gk}
 G_k := \frac{ (\tilde{\bar g}^{(k)})^{1/2} \circ \bar w_k \,\,  d\bar w_k \, \,  (\bar g^{(k)})^{-1/2} - \Id}{h_k^2}  \rightharpoonup G  \quad \text{in $L^2(B_1(0); \R^{n \times n})$}
\end{equation}
with 
\begin{equation}   \label{eq:formual_limit_G}
G_{im}(x)  =  d(\imath_{\underline V}^{-1} \circ \check f \circ  \imath_{\underline V})(x)  -  \sum_{j,l=1}^n \mathcal A^i_{jml} x^j x^l, 
\quad \text{and} \quad \mathcal A^i_{jml} =  \frac16  g(p)(V_i, (\mathcal R(p) - \tR^Q)(V_j, V_m, V_l)).
\end{equation}

Now  set $F_k := \{ x \in B_1(0) : |h_k^2 G_k| > h_k \}$ 
and for $x \notin F_k$ use the  Taylor expansion 
$$\dist^2(\Id + h_k^2 G_k, SO(n)) = |\sym h_k^2 G_k|^2 + \mathcal O(h_k) |h_k^2 G_k|^2.$$
By \eqref{eq:metric_normal} we have $d\Vol_{\varphi_k^*g} = h^n (1 + \mathcal O(h_k^2)) \mathcal L^n$.
 Using that $\mathcal L^n(\bar E_k \cup F_k) \to 0$ and   that positive semidefinite quadratic
forms are weakly lower semi-continuous,  we deduce that
\begin{eqnarray} 
  & & \liminf_{k \to \infty} \frac{1}{h_k^4} \fint_{B_1(0)}  (1_{B_1(0) \setminus \bar E_k)} \dist^2( \Id + h_k^2 G_k)  \, d\Vol_{\varphi^*g} \nonumber \\
& \ge & \liminf_{k \to \infty} \fint_{B_1(0)}  |(1_{B_1(0) \setminus (\bar E_k \cup F_k)} \sym G_k|^2 \, dx \nonumber \\
& \ge & \fint_{B_1(0)} |\sym G|^2 \, dx.   \label{eq:weaklsc}
\end{eqnarray}
Now the assertion follows from   \eqref{eq:lower_bound_Eh_wk},  \eqref{eq:dwk_local_compactness2},
 \eqref{eq:define_Gk}, 
\eqref{eq:formual_limit_G} and \eqref{eq:weaklsc}.
\end{proof}

\begin{proof}[Proof of Theorem~\ref{th:main}~\eqref{it:main3} (recovery sequence)]
Let  $ q \in \tM$,  $Q \in SO(n)(T_p \mathcal M, T_q \tM)$ and let $\check f \in W^{1,2}(B_1(0), T_p \mathcal M)$ be a representative of $f$. There exists Lipschitz maps
$\check f_k$ such that $\check{f}_k \to \check f$ in $W^{1,2}$ and $\Lip \check f_k \le h_k^{-1}$.
Set  $c_k = - h_k^2 \fint_{B_1(0)} \check f_k$ and define
\begin{equation} \label{eq:define_recovery}
 u_k(\exp_p (h_k X)) :=  w_k(\exp_p (h_k X)) := \exp_q (h_k Q (X + h_k^2 \check f_k(X) +   c_k)).
 \end{equation}
 Then \eqref{eq:convergence_condition_1}--\eqref{eq:convergence_condition_2}  hold,   and   the definition 
  \eqref{eq:define_qk} of $q_k$ in combination with the definition 
  of $c_k$ implies  that $q_k = q$. 
 The definition \eqref{eq:define_fk} of $f_k$ with the choice $Q_k = Q$  yields
$ f_k =  \check f_k$. Thus $u_k \to (q,Q,f)$. 

To show convergence of the rescaled energy,  we define $G_k$ as in  \eqref{eq:define_Gk} and \eqref{eq:define_barwk_gammaliminf}, with the frame $\underline{\tilde V}_k = (QV_1, \ldots, QV_n)$ in the target space
(recall that $q_k = q$ and $Q_k = Q$).
Then $G_k \to G$ in $L^2$ (strongly), with $G$ given by    \eqref{eq:formual_limit_G}.
Since $\dist^2(F, SO(n)) \le C|F - \Id|^2$  and $\int_{|G_k| \ge h_k^{-1}} |G_k|^2 \, dx \to 0$, Taylor expansion shows that 
$$  \lim_{k \to \infty} \frac{1}{h_k^4} \fint_{B_1(0)}  \dist (\Id + h_k^2 G_k, SO(n)))   \, dx = \fint_{B_1(0)} |\sym G|^2 \, dx.$$
In view of  \eqref{eq:dwk_local_compactness2} and the choice $u_k = w_k$,  we get the desired assertion.
\end{proof}

\section{Convergence of the energy}  \label{se:energy}

It is easy to see that the quadratic functional $f \mapsto \mathcal I^{q,Q}(f)$ attains its minimum 
in $W^{1,2}(B_1(0), T_p \mathcal M)$. Set
\begin{equation}
m^{q,Q} := \min_{ f \in W^{1,2}(B_1(0), T_p \mathcal M)} \mathcal I^{q,Q}(f).
\end{equation}

\begin{theorem} Let $\tM$ be compact. Then 
\begin{equation}  \label{eq:main_convergence}
\lim_{h \to 0}  \frac{1}{h^4}  \inf_{ u \in W^{1,2}(B_h(p); \tM)} E_{B_h(p)}(u) =  \bar m := \min_{q \in \tM}   \, \, \min_{Q \in SO(T_p \mathcal M, T_q\tM)} m^{q,Q}.
\end{equation}
\end{theorem}

\begin{proof} This is a standard consequence of Theorem~\ref{th:main}. We include the details for the convenience of the reader.

 It is easy to see that the map ${q,Q} \mapsto m^{q,Q}$ is continuous as a map from the subbundle $SO(T_p \mathcal M, T\tM) \subset T\tM \otimes T_p^* \mathcal M$ to $\R$.
Since $\tM$ is compact, so is $SO(T_p \mathcal M, T\tM)$. Thus the minimum on the right hand side of \eqref{eq:main_convergence} exists.

Upper bound: set
$L^+ = \limsup_{h \to 0}  h^{-4}  \inf_{ u \in W^{1,2}(B_h(p); \tM)} E_{B_h(p)}(u)$ and let $h_k \to 0$ be a subsequence along which the limit superior
is realised.  Let  $q \in \tM$, $ Q \in SO(T_p \mathcal M, T_q \tM)$,  and let $f$ be a  minimiser of $\mathcal I^{q,Q}$. It follows from  Theorem~\ref{th:main}~\eqref{it:main3} that
$L^+ \le m^{q,Q}$.  Optimising over $Q$ and $q$, we get $L^+ \le \bar m$.

Lower bound: set $L^- = \liminf_{h \to 0}  h^{-4}  \inf_{ u \in W^{1,2}(B_h(p); \tM)} E_{B_h(p)}(u)$ and let $h_k \to 0$ be a subsequence which realises the 
limit inferior. Then there exist maps $u_k$ such that 
$$ \lim_{k \to \infty} \frac{1}{h_k^4}  E_{B_{h_k}(p)}(u_k) = L^-.$$
By Theorem~\ref{th:main}~\eqref{it:main1} there exists a subsequence $u_{k_j}$ which converges to $(q,Q,f)$ in the 
sense of Definition~\ref{def:define_convergence_notion}.
Thus Theorem~\ref{th:main}~\eqref{it:main2} implies that
$ L^- \ge \mathcal I^{q,Q}(f) \ge m^{q,Q} \ge \bar m$.
\end{proof}

A slight modification of the arguments in the proof of  Theorem~\ref{th:main} yields the following extension for non-compact targets.

\begin{corollary} \label{co:noncompact_target}
Suppose that $\tM$ is complete and satisfies the following uniform regularity condition: there exists a $\rho > 0$
such that the injectivity radius satisfies   $\inj(q) \ge \rho$ for all $q \in \tM$  and the the pullback metrics $\exp_q^* g$ are uniformly bounded in $C^3(B_\rho(0))$. 
Then 
\begin{equation}  \label{eq:main_convergence_noncompact}
\lim_{h \to 0}  \frac{1}{h^4}  \inf_{ u \in W^{1,2}(B_h(p); \tM)} E_{B_h(p)}(u) =  \inf_{q \in \tM}   \, \, \min_{Q \in SO(T_p \mathcal M, T_q\tM)} m^{q,Q}.
\end{equation}
\end{corollary}

\section*{Acknowledgements} 
The authors thank Cy Maor for very helpful suggestions and for pointing out reference \cite{kupferman_maor_asaf_reshetnyak}.
This work is an extension of the first author's B.Sc.~thesis at the University of Bonn. 
In that thesis a recovery sequence is constructed, and  compactness and the $\Gamma-\liminf$ inequality are shown
under the additional hypothesis that the original sequence $u_k$ satisfies a uniform Lipschitz bound and $u_k(p)$ is fixed.
The second author has been supported by the Deutsche Forschungsgemeinschaft (DFG, German Research Foundation) through
the Hausdorff Center for Mathematics (GZ EXC 59 and 2047/1, Projekt-ID 390685813) and the 
collaborative research centre  {\em The mathematics of emerging effects} (CRC 1060, Projekt-ID 211504053).

\end{document}